\documentclass[12pt,oneside,english]{amsart}
\usepackage[T1]{fontenc}
\usepackage[latin9]{inputenc}
\usepackage{amsthm}
\usepackage{amsmath}
\usepackage{amssymb}
\usepackage{mathrsfs}

\makeatletter
\numberwithin{equation}{section}
\numberwithin{figure}{section}

\usepackage{geometry}
\usepackage[svgnames]{xcolor}
\usepackage[bookmarksnumbered=true]{hyperref} 
\hypersetup{
	colorlinks = true,
	linkcolor = Blue,
	anchorcolor = blue,
	citecolor = Green,
	filecolor = blue,
	urlcolor = FireBrick
}
\usepackage{bbm}

\usepackage{comment}

\DeclareRobustCommand{\SkipTocEntry}[5]{}

\newcommand{\mR}{\mathbb{R}}   
\newcommand{\abs}[1]{\lvert #1 \rvert}  
\newcommand{\norm}[1]{\lVert #1 \rVert}  

\newcommand{\mD}{\mathscr{D}}

\newcommand{\mH}{\mathscr{H}}

\newcommand{\bfV}{\mathbf{V}}

\newcommand{\mL}{\mathcal{L}}

\newcommand{\bfi}{\mathbf{i}}

\newcommand{\vareps}{\varepsilon}

\newcommand{\p}{\partial}

\DeclareMathOperator{\dist}{dist}

\newcommand{\rmd}{\mathrm{d}}

\newcommand{\Id}{\mathrm{Id}}

\numberwithin{equation}{section}
\numberwithin{figure}{section}
\theoremstyle{plain}
\newtheorem{theorem}{\protect\theoremname}[section]
\theoremstyle{definition}

\theoremstyle{definition}
\newtheorem{example}[theorem]{\protect\examplename}
\theoremstyle{definition}

\theoremstyle{plain}
\newtheorem{lemma}[theorem]{\protect\lemmaname}
\theoremstyle{plain}
\newtheorem{corollary}[theorem]{\protect\corollaryname}
\theoremstyle{plain}
\newtheorem{proposition}[theorem]{\protect\propositionname}
\theoremstyle{plain}

\newtheorem*{remark*}{Remark}
\newtheorem*{definition*}{Definition}

\usepackage{graphicx}

\usepackage{babel}
\providecommand{\corollaryname}{Corollary}
\providecommand{\definitionname}{Definition}
\providecommand{\examplename}{Example}
\providecommand{\lemmaname}{Lemma}
\providecommand{\propositionname}{Proposition}
\providecommand{\remarkname}{Remark}
\providecommand{\theoremname}{Theorem}
\providecommand{\conjecturename}{Conjecture}

\makeatother

\usepackage{babel}
\begin{document}
\begin{sloppypar}

\title{On scattering behavior of corner domains with anisotropic inhomogeneities}

\author{Pu-Zhao Kow}
\address{Department of Mathematical Sciences, National Chengchi University, No. 64, Sec. 2, ZhiNan Rd., Wenshan District, 116302 Taipei, Taiwan}
\email{pzkow@g.nccu.edu.tw}

\author{Mikko Salo}
\address{Department of Mathematics and Statistics, P.O. Box 35 (MaD), FI-40014 University of Jyv\"{a}skyl\"{a}, Finland}
\email{mikko.j.salo@jyu.fi}

\author{Henrik Shahgholian}
\address{Department of Mathematics, KTH Royal Institute of Technology, SE-10044 Stockholm, Sweden}
\email{henriksh@kth.se}

\begin{abstract}
This paper investigates the possible scattering and non-scattering behavior of an anisotropic and inhomogeneous Lipschitz medium at a fixed wave number and with a single incident field. We connect the anisotropic non-scattering problem to a Bernoulli type free boundary problem. By invoking methods from the theory of free boundaries, we show that an anisotropic medium with Lipschitz but not $C^{1,\alpha}$ boundary scatters every incident wave that satisfies a non-degeneracy condition.
\end{abstract}


\maketitle

\section{Introduction}

\subsection{Background}
We investigate the  problem of unraveling the nature of scattered waves, wherein the obstructing medium is a bounded region, and the irregularities within it are described by coefficients that may exhibit anisotropic properties. The scattering problem is modelled by the following wave equation:
\[
c(x)^{-2} \p_t^2 U - \nabla \cdot (A(x) \nabla U) = 0 \text{ in $\mR^n \times \{ t > 0\}$}.
\]
Here, the velocity of sound, denoted as $c$, and the symmetric matrix $A$ are in $L^{\infty}(\mathbb{R}^n)$, exhibiting uniform lower bounds throughout the medium. Notably, this equation encompasses both the classical wave equation, $c^{-2} \partial_t^2 U - \Delta U = 0$, where the sound speed is scalar, as well as the Riemannian wave equation, $\partial_t^2 U - \Delta_g U = 0$, which involves a Riemannian metric $g$, by making appropriate choices.

We consider scattering of waves with fixed frequency $\kappa > 0$,  which  corresponds to solutions of the form $U(x,t) = e^{\bfi \kappa t} u^{\rm to}(x)$, where $u^{\rm to}$ satisfies 
\[
\nabla \cdot (A(x) \nabla u^{\rm to}) + \kappa^2 \rho(x) u^{\rm to} = 0 \text{ in $\mR^n$}
\]
with $\rho = c^{-2}$. If we probe the medium with an incoming wave $u^{\rm inc}$ that solves 
\begin{equation}
(\Delta + \kappa^{2})u^{\rm inc} = 0 \quad \text{in $\mR^{n}$,} \label{eq:incident-field}
\end{equation}
then the total wave $u^{\rm to}$ has the form $u^{\rm to} = u^{\mathrm{inc}} + u^{\mathrm{sc}}$ where the scattered wave $u^{\mathrm{sc}}$ satisfies the outgoing Sommerfeld radiation condition.

Now, we proceed to provide a detailed mathematical expression. Consider $\Omega$, a bounded region in $\mathbb{R}^{n}$ (where $n\geq 2$) with a Lipschitz boundary and with $\mathbb{R}^n \setminus \overline{\Omega}$ connected. Within this domain, let $\rho \in L^{\infty}(\Omega)$ be a positive real-valued function. Additionally, let $A \in (C^{0,1}(\overline{\Omega}))^{n \times n}$ be a real symmetric matrix-valued function, satisfying the condition of uniform ellipticity
\begin{equation}
c_{\rm ellip}^{-1}\abs{\xi}^{2} \le \xi \cdot A(x)\xi \le c_{\rm ellip}\abs{\xi}^{2} \quad \text{for a.e. $x\in\Omega$ and all $\xi\in\mR^{n}$} \label{eq:unifurm-ellipticity}
\end{equation} 
for some constant $c_{\rm ellip}>0$.

Under the assumption that the medium outside $\Omega$ is homogeneous, if we illuminate the anisotropic medium $(\Omega, A, \rho)$ with an incident field $u^{\rm inc}$ having a fixed wave number $\kappa>0$ that satisfies \eqref{eq:incident-field}, classical scattering theory (see e.g. \cite[Theorem~1.38]{CCH23InverseScatteringTransmission} or \cite{CK19scattering,KG08Factorization}) guarantees the existence of unique scattered field $u^{\rm sc} \in H_{\rm loc}^{1}(\mR^{n})$ which is outgoing (the fact that $A$ is Lipschitz is required here since the argument involves the unique continuation principle). The total field $u^{\rm to} = u^{\rm sc} + u^{\rm inc}$ satisfies the following condition
\begin{equation*}
\left( \nabla \cdot \tilde{A}(x) \nabla + \kappa^{2} \tilde{\rho}(x) \right) u^{\rm to} = 0 \quad \text{in $\mR^{n}$,}
\end{equation*}
where 
\begin{equation}\label{eq:A}
\tilde{A} = A \chi_{\Omega} + \Id \chi_{\mR^{n}\setminus\Omega}, \quad \hbox{and} \quad  \tilde{\rho} = \rho \chi_{\Omega} + \chi_{\mR^{n}\setminus\Omega}.
\end{equation}
We recall the following definition. 

\begin{definition*}
A solution $v$ of $(\Delta + \kappa^{2})v=0$ in $\mR^{n}\setminus B_{R}$ (for some $R>0$) is outgoing if it satisfies the following Sommerfeld radiation condition: 
\begin{equation*}
\lim_{\abs{x}\rightarrow \infty} \abs{x}^{\frac{n-1}{2}} (\partial_{\abs{x}} v - \bfi \kappa v) = 0,\quad \text{uniformly in all directions $\hat{x}=\frac{x}{\abs{x}} \in \mathcal{S}^{n-1}$,}
\end{equation*}
where $\partial_{\abs{x}} = \hat{x} \cdot \nabla$ denotes the radial derivative. In this case, the far-field pattern $v^{\infty}$ of $v$ is defined by 
\begin{equation*}
v^{\infty}(\hat{x}) := \lim_{\abs{x}\rightarrow \infty} \gamma_{n,\kappa}^{-1} \abs{x}^{\frac{n-1}{2}} e^{-\bfi\kappa\abs{x}} v(x)  \quad \text{for all $\hat{x} \in \mathcal{S}^{n-1}$}
\end{equation*}
for some normalizing constant $\gamma_{n,\kappa} \neq 0$. 
\end{definition*}

The Rellich uniqueness theorem \cite{CK19scattering,Hormander_rellich} implies that
\begin{equation*}
v^{\infty} \equiv 0 \quad \text{if and only if} \quad v=0 \text{ in $\mR^{n}\setminus\overline{\Omega}$.}
\end{equation*}
We are interested in the following question: does the anisotropic medium $(\Omega, A, \rho)$ scatter every incoming wave nontrivially, or can there be some incoming wave that produces no scattering (i.e.\ $u^{\rm sc}$ has zero far-field pattern)? The rigorous analysis of this phenomenon was initiated for $A = \mathrm{Id}$ in \cite{BPS14CornerScattering}, which showed that corners in the scattering obstacle $\Omega$ might always scatter every incoming wave nontrivially. Similar corner scattering results and related single measurement uniqueness results have been proved in various other settings (see e.g.\ \cite{HSV16SingleMeasurement, PSV17CornerScattering, ElschnerHu18, BlastenLiu21} and the survey \cite{Liu22}). The works \cite{CV21SingularitiesAlwaysScatter, SS21NonscatteringFreeBoundary} introduced powerful new methods from free boundary problems to this setting, allowing one to deal with obstacles with Lipschitz or less regular boundaries. The anisotropic case was studied in \cite{CVX23RegularityITEP}.

The main feature of this work is to show that the anisotropic non-scattering problem can be related to a Bernoulli problem in free boundary theory. We will use methods from Bernoulli problems to improve the results in \cite{CVX23RegularityITEP} to the case of obstacles with Lipschitz boundaries, thus covering the case of actual corners.

More precisely, if the anisotropic medium $(\Omega,A,\rho)$ is non-scattering with respect to the incident field $u^{\rm inc}$ in the sense of $u^{\rm sc}=0$ in $\mR^{n}\setminus\overline{\Omega}$, then the pair $(u^{\rm inc},u^{\rm to}) \in H_{\rm loc}^{1}(\mR^{n}) \times H_{\rm loc}^{1}(\mR^{n})$ satisfies the following problem (similar to the interior transmission eigenvalue problem):
\begin{equation}
\begin{cases}
\left( \mL + \kappa^{2}\rho(x) \right) u^{\rm to} = 0, \quad (\Delta + \kappa^{2})u^{\rm inc} = 0, & \text{in $\Omega$,} \\
u^{\rm to} = u^{\rm inc} , \quad \nu \cdot A(x)\nabla u^{\rm to} = \partial_{\nu} u^{\rm inc}, & \text{on $\partial \Omega$,}
\end{cases} \label{eq:ITEP}
\end{equation} 
where $\mL = \nabla\cdot A(x) \nabla$, $\nu$ is the \emph{inward} unit normal vector to $\partial\Omega$ (we choose this orientation for later convenience) and $\partial_{\nu} = \nu\cdot \nabla$ is the normal derivative in the sense of \cite[Theorem~5.8.1]{EG15MeasureTheory}. One also sees that the scattered field $u^{\rm sc} := u^{\rm to} - u^{\rm inc} \in H_{\rm loc}^{1}(\mR^{n})$ satisfies 
\begin{equation}
\begin{cases}
(\mL + \kappa^{2}\rho(x)) u^{\rm sc} = -(\mL + \kappa^{2}\rho(x)) u^{\rm inc} & \text{in $\Omega$,} \\
u^{\rm sc}=0 ,\quad \nu \cdot A \nabla u^{\rm sc} = \nu \cdot (\Id - A) \nabla u^{\rm inc} & \text{on $\partial\Omega$.}
\end{cases} \label{eq:ITEP-scattering}
\end{equation}

The equation presented in \eqref{eq:ITEP-scattering} portrays a classical instance of a free boundary problem known as the Bernoulli type, which has garnered attention over the course of numerous decades from diverse vantage points. Of specific relevance to our inquiry is the examination of particular outcomes, with a focus on the smoothness of $\partial \Omega$ under certain a priori smoothness assumptions, such as Lipschitz continuity. This constitutes the central subject matter of the present paper.

\subsection{Main results} \label{sec_main}

Now we state our main results.

\begin{theorem}\label{THM:main1}
Let $\Omega$ be a bounded Lipschitz domain in $\mathbb{R}^{n}$ $(\text{where } n \ge 2)$, let $\rho \in L^{\infty}(\Omega)$ be a positive real-valued function, and let $A \in (C^{0,1}(\overline{\Omega}))^{n\times n}$ be a real symmetric matrix-valued function satisfying the condition of uniform ellipticity \eqref{eq:unifurm-ellipticity}. 
Suppose that the anisotropic medium $(\Omega,A,\rho)$ is non-scattering with respect to $u^{\rm inc}$ in the sense of \eqref{eq:ITEP}. For $x_{0} \in \partial\Omega$, suppose that $A$ has a $C^{1}$-extension near $x_{0}$ and suppose that $\rho$ has a $C^0$-extension near $x_{0}$. Suppose further that one of the following non-degeneracy conditions holds:
\begin{equation}
\begin{cases}
\text{$\nu \cdot (\rm Id - A) \nabla u^{\rm inc} \ge c > 0$ Hausdorff-a.e.\ on $\partial\Omega$ near $x_{0}$; or} \\
\text{$\nu \cdot (\rm Id - A) \nabla u^{\rm inc} \le -c < 0$ Hausdorff-a.e.\ on $\partial\Omega$ near $x_{0}$.}
\end{cases} \label{EQ:non-degeneracy-condition}
\end{equation}
Then $u^{\rm sc}$ is Lipschitz continuous and $\partial\Omega$ is $C^{1,\alpha}$ near $x_{0}$. 
\end{theorem}

We remark that if $\partial\Omega$ is $C^{1}$ near $x_{0}$, then the normal vector $\nu$ defines a continuous vector field on $\partial\Omega$ near $x_{0}$. In this case, \eqref{EQ:non-degeneracy-condition} can be replaced by 
\begin{equation*}
\nu \cdot (\rm Id - A) \nabla u^{\rm inc}(x_0) \neq 0.
\end{equation*}
The above result shows that if 
$\p \Omega$  is not $C^{1,\alpha}$ near $x_0$ and if the non-degeneracy condition \eqref{EQ:non-degeneracy-condition} holds, then the obstacle scatters $u^{\mathrm{inc}}$ non-trivially. We can summarize the above result as ``corners conditionally always scatter'', compare to \cite{BPS14CornerScattering} and subsequent works.

Combining our result with \cite[Theorem~2.1]{CVX23RegularityITEP}, we conclude the following corollary. 

\begin{corollary}
Suppose that all assumptions in Theorem~{\rm \ref{THM:main1}} hold. If we further assume 
\begin{equation*}
A \in (C^{\ell+1,\alpha}(\overline{\Omega}))^{n\times n} \text{ and } \rho \in C^{\ell,\alpha}(\overline{\Omega})
\end{equation*}
for some $\ell \in \mathbb{N}$, then $\partial\Omega$ is $C^{\ell+1,\alpha}$ near $x_{0}$. In addition, if $A$ and $\rho$ are both smooth $(\text{resp.\ real analytic})$ in $\overline{\Omega}$, then $\partial\Omega$ is smooth $(\text{resp.\ real analytic})$ near $x_{0}$. 
\end{corollary}

We can also give an application to radiating and nonradiating sources. The investigation of such sources -- for acoustic, electromagnetic and elastic waves -- has a long history, see e.g.\ \cite[Section~2.3]{KW21CharacterizeNonradiating} for related works. 
We say that the pair $(g,h) \in H^{-\frac{1}{2}}(\partial\Omega) \times L^{2}(\Omega)$ is a \emph{nonradiating source} if the unique outgoing solution $w \in H_{\rm loc}^{1}(\mR^{n})$ satisfies $w = 0$ in $\mR^{n}\setminus\overline{\Omega}$, more precisely, 
\begin{equation}
\left\{\begin{aligned}
&(\mL + \kappa^{2}\rho(x))w = h && \text{in $\Omega$,} \\
&w = 0 && \text{in $\mR^{n}\setminus\overline{\Omega}$,} \\
&(\partial_{\nu} w)_{\rm int} = g && \text{on $\partial\Omega$.} 
\end{aligned}\right. \label{eq:transmission-problem-nonradiating}
\end{equation}
Here we also point out that the interior transmission eigenvalue problem considered in \cite{DDL22ITEP} can be written in the form \eqref{eq:transmission-problem-nonradiating} for some suitable $(g,h)$, see also \eqref{eq:transmission-problem} below. We have the following theorem. 

\begin{theorem}\label{THM:non-radiating-source1}
Let $\Omega$ be a bounded Lipschitz domain in $\mathbb{R}^{n}$ $(\text{where } n \ge 2)$, let $h \in L^{\infty}(\Omega)$ be a positive real-valued function, and let $A \in (C^{0,1}(\Omega))^{n\times n}$ be a real symmetric matrix-valued function satisfying the condition of uniform ellipticity \eqref{eq:unifurm-ellipticity}. Let $w$ solve the system \eqref{eq:transmission-problem-nonradiating}. For $x_{0} \in \partial \Omega$, suppose that $A$ has a $C^{1}$-extension near $x_{0}$ and suppose that $h$ has a $C^0$-extension near $x_{0}$. Assume that $g = \nu \cdot A \bfV$,  for some  Lipschitz continuous vector field $\bfV$ that is transversal to $\partial\Omega$ and satisfies 
\begin{equation*}
\nu \cdot A \bfV \ge c_{3} \text{ in a region of $\Omega$ near $x_{0}$} \quad \text{or} \quad \nu \cdot A \bfV \le - c_{3} \text{ in a region of $\Omega$ near $x_{0}$.}
\end{equation*}
Then the function $w$ is Lipschitz continuous, and the boundary $\partial\Omega$ is $C^{1,\alpha}$ near $x_{0}$. 
\end{theorem}

The above theorems will be proven in Section \ref{Sec:2}. Here $w$ plays the role of $u^{\rm sc}$ in \eqref{eq:ITEP-scattering}. 

\subsection{Further directions}

As in \cite{CVX23RegularityITEP}, we will now consider examples of domains that are non-scattering for some incident waves in the sense discussed above.

\begin{example} \label{example1}
Let $\Omega$ be a bounded domain in $\mathbb{R}^2$, and suppose that $\kappa > 0$ is such that there is a nontrivial global solution $w$ of $(\Delta + \kappa^2)w=0$ in $\mathbb{R}^2$ with $w|_{\partial \Omega} = 0$. (With minor modifications one could also work with $\partial_{\nu} w|_{\partial \Omega} = 0$.) Then necessarily $\kappa^2$ is a Dirichlet eigenvalue of $-\Delta$ in $\Omega$. If $a \neq 1$ is a constant and if we take $A = a \, \mathrm{Id}$ and $\rho=a$, then $u = w$ and $v = aw$ satisfy the analogue of \eqref{eq:ITEP}:
\begin{equation}
\begin{cases}
\left( \mL + \kappa^{2}\rho(x) \right) u = 0, \quad (\Delta + \kappa^{2})v = 0, & \text{in $\Omega$,} \\
u = v , \quad \nu \cdot A(x)\nabla u = \partial_{\nu} v, & \text{on $\partial \Omega$.}
\end{cases} \label{eq:ITEP_polygon}
\end{equation} 
Thus the isotropic medium $(\Omega, A, \rho)$ is non-scattering for the incident wave $w$. 

In \cite[Section 3]{CVX23RegularityITEP} one chose $\Omega = (0,1)^2$ to be the unit square and $w(x) = \sin(p\pi x_1) \sin(q \pi x_2)$ for $p, q \in \mathbb{Z} \setminus \{0\}$. Let us show that one can have such non-scattering domains with corners of angle $\ell \pi/m$ for any integers $m \geq 2$ and $1 \leq \ell < 2m-1$. (The angles must be of this form since the zero set of a nontrivial solution $w$ of a second order elliptic equation in $\mathbb{R}^2$ is locally the union of $m$ curves that intersect at angles $\pi/m$, see e.g.\ \cite{LogunovMalinnikova2020}.) Let $(r,\theta)$ be polar coordinates in $\mathbb{R}^2$ with $(x_{1},x_{2})=(r\cos\theta,r\sin\theta)$ and let $\Omega = \{ \, 0 < r < 1, \ 0 < \theta < \ell \pi/m \}$ be a sector domain. The eigenfunctions of the Laplacian on $\Omega$ are known \cite{GrebenkovNguyen2013}. Let $(\alpha_k)$ be the positive zeros of the Bessel function $J_m$, and define 
\[
w(r,\theta) = J_m(\alpha_k r) \sin(m\theta).
\] 
Writing $z = re^{i\theta}$ we have $w(z) = |z|^{-m} J_m(\alpha_k |z|) \mathrm{Im}(z^m)$, which is a smooth function in $\mathbb{R}^2$ by properties of $J_m$. One also has $(\Delta + \alpha_k^2) w = 0$ in $\mathbb{R}^2$ and $w|_{\partial \Omega} = 0$. It follows that $\Omega$ is a non-scattering domain for the incident wave $w$. 

The fact that such non-scattering corner domains exist does not contradict Theorem \ref{THM:main1}, since $\nabla w(x_i) = 0$ at each corner point $x_i$ of $\Omega$ (with $w$ having a zero of order $m$ at $0$) and hence the incident wave $w$ does not satisfy the non-degeneracy condition \eqref{EQ:non-degeneracy-condition}. We also note that in this example both $A$ and $\rho$ have a jump at $\partial \Omega$. If only $\rho$ has a jump but $A$ does not, non-scattering corner domains may not exist in $\mathbb{R}^{n}$ e.g.\ by \cite{ElschnerHu18, CX21CornerScatteringEllipticOperator}.

To study the Bernoulli condition satisfied by $w$, we compute $\nabla w$ on $\Gamma := \partial\Omega \setminus \{r=1\}$. By direct computations, one has 
\begin{equation*}
\partial_{r}w = \alpha_{k} J_{m}'(\alpha_{k}r) \sin(m\theta) ,\quad \partial_{\theta}w = mJ_{m}(\alpha_{k}r)\cos(m\theta).
\end{equation*}
It is easy to see that $\left. \partial_{r}w \right|_{\Gamma}=0$. By writing $\Gamma_{0}=\Gamma\cap\{\theta=0\}$ and $\Gamma_{\ell\pi/m}=\Gamma\cap\{\theta=\ell\pi/m\}$, we see that (recall that $\nu$ is pointing inward to $\Omega$)
\begin{equation*}
\begin{aligned}
& \left. \partial_{\theta}w \right|_{\Gamma_{0}} = m J_{m}(\alpha_{k}r) ,\quad \left. \nu \right|_{\Gamma_{0}} = (0,1), \\
& \left. \partial_{\theta}w \right|_{\Gamma_{\ell\pi/m}} = m J_{m}(\alpha_{k}r)(-1)^{\ell} ,\quad \left. \nu \right|_{\Gamma_{\ell\pi/m}} = \left( \sin \frac{\ell\pi}{m} ,-\cos \frac{\ell\pi}{m}  \right).
\end{aligned}
\end{equation*}
Since  
\begin{equation*}
\partial_{x_{1}}w = \cos\theta \partial_{r}w - \frac{\sin\theta}{r}\partial_{\theta}w ,\quad \partial_{x_{2}}w = \sin\theta \partial_{r}w + \frac{\cos\theta}{r}\partial_{\theta}w,
\end{equation*}
then 
\begin{equation*}
\left. \partial_{\nu}w \right|_{\Gamma_{0}} = \left. \partial_{x_{2}}w \right|_{\Gamma_{0}} = \left. \frac{1}{r}\partial_{\theta}w \right|_{\Gamma_{0}} = \frac{m}{\abs{x}}J_{m}(\alpha_{k}\abs{x}) \sim \frac{\alpha_{k}^{m}}{2^{m}(m-1)!}\abs{x}^{m-1} \quad \text{near $x=0$,}
\end{equation*}
and 
\begin{equation*}
\begin{aligned}
& \left. \partial_{\nu}w \right|_{\Gamma_{\ell\pi/m}} = \sin \frac{\ell\pi}{m} \left. \partial_{x_{1}}w  - \cos \frac{\ell\pi}{m}\partial_{x_{2}}w \right|_{\Gamma_{\ell\pi/m}} = - \left. \frac{1}{r}\partial_{\theta}w \right|_{\Gamma_{\ell\pi/m}} \\
& \quad = (-1)^{\ell+1}\frac{m}{\abs{x}}J_{m}(\alpha_{k}\abs{x}) \sim (-1)^{\ell+1} \frac{\alpha_{k}^{m}}{2^{m}(m-1)!}\abs{x}^{m-1} \quad \text{near $x=0$.}
\end{aligned}
\end{equation*}
Then the Bernoulli boundary condition on $\partial\Omega$ near the origin is  
\begin{equation}
\abs{\nabla w(x)} = \frac{m}{\abs{x}}\abs{J_{m}(\alpha_{k}\abs{x})} \sim \frac{\alpha_{k}^{m}}{2^{m}(m-1)!} \abs{x}^{m-1} \quad \text{for all $x\in\partial\Omega$ near $x=0$.} \label{eq:Bernuolli-example}
\end{equation}
Moreover, $\left. \partial_{\nu}w \right|_{\partial\Omega}$ does not change sign near $x=0$ when $\ell$ is odd, but it changes sign when $\ell$ is even.
\end{example}

\begin{example} \label{example2}
We will now consider the other example in \cite[Section 3]{CVX23RegularityITEP} based on diffeomorphism invariance. Let $\Omega \subset \mathbb{R}^n$ be a bounded domain, and let $\Phi: \Omega \to \Omega$ be a diffeomorphism such that $\Phi$ and $\Phi^{-1}$ extend smoothly to $\overline{\Omega}$ and $\Phi(x) = x$ for $x \in \partial \Omega$. Let $A = \Phi_*(\mathrm{Id})$ and $\rho = \Phi_*(1)$ be the pushforwards by $\Phi$. Then $v$ solves $(\Delta+\kappa^2)v =0$ in $\Omega$ if and only if $u = \Phi_* v$ solves $(\mL + \kappa^2 \rho)u =0$ in $\Omega$. If $w \not\equiv 0$ solves $(\Delta+\kappa^2)w = 0$ in $\mathbb{R}^n$, then choosing $v = w|_{\Omega}$ and $u = \Phi_* v$ gives a pair $(u,v)$ satisfying \eqref{eq:ITEP_polygon}. Hence $(\Omega,A,\rho)$ is non-scattering for the incident wave $w$.

Suppose that $\partial\Omega$ is piecewise smooth. In this case, the condition $\Phi(x) = x$ for $x \in \partial \Omega$ implies that $D\Phi=\Id$ at the corners of $\partial\Omega$. Then $\left. (\Id-A(x))\nabla w \right|_{\partial\Omega} \equiv 0$ at the corners, so the non-degeneracy condition \eqref{EQ:non-degeneracy-condition} is always violated in such a setting.
\end{example}

Let us compare the above two examples. In Example \ref{example1} the functions $u$ and $v$ came from a function $w$ that solves an elliptic equation near the corner and satisfies the additional condition $w|_{\partial \Omega} = 0$. This additional condition forced the angle of the corner to be a rational multiple of $\pi$. On the other hand, in Example \ref{example2} the functions $u$ and $v$ came from a solution $w$ that was not required to vanish on $\partial \Omega$, and thus the angle of the corner could be any real number. 

Both examples above are related to solutions of a Bernoulli problem. To further explain this point, the next example gives another solution of a Bernoulli problem for the Laplacian where the domain can have a corner of arbitrary angle.

\begin{example}
Let $\alpha > 1/2$ and let $w = \mathrm{Re}(z^{\alpha}) = \mathrm{Re}(e^{\alpha \log z})$ where $\log z$ is the principal branch of the complex logarithm. If $\Omega = \{ r e^{i\theta} \mid r > 0, \, -\frac{\pi}{2\alpha} < \theta < \frac{\pi}{2\alpha} \}$, then $w$ satisfies the Bernoulli problem
\[
\Delta w = 0 \text{ in $\Omega$}, \quad w|_{\partial \Omega} = 0, \quad \p_{\nu} w|_{\{\theta=\pm \pi/(2\alpha)\}} = \alpha r^{\alpha-1}.
\]
One sees that $\p_{\nu} w|_{\partial\Omega}$ vanishes (resp.\ blows up) at 0 of order $\alpha-1$ when $\alpha>1$ (resp.\ $1/2<\alpha<1$). Note that 0 is a corner of $\Omega$ with angle $\pi/\alpha$ when $\alpha \neq 1$.
Of course, $w$ can only be extended as a solution near $0$ when $\alpha$ is an integer (in this case, $\Omega$ has a corner whose angle is a rational multiple of $\pi$).
\end{example}

The concepts employed in proving Theorem~{\rm \ref{THM:main1}} are equally applicable for examining a specific category of transmission problems, linked to the two-phase Bernoulli problem (as discussed in \cite{ACF84TwoPhasesFreeBoundary} or the comprehensive reference \cite{CS05FreeBoundary}). Consider a positive real-valued function $\rho$ within the space $L^{\infty}(\Omega)$. By extending the concepts from \cite[Theorem~2.2.1]{Bondarenko16thesis}, it can be demonstrated that, for each $0 \le \lambda \in L^{\infty}(\partial\Omega)$, $h \in L^{2}(\Omega)$, and $g \in H^{-\frac{1}{2}}(\partial\Omega)$, there exists a unique \emph{outgoing} solution $w \in H_{\rm loc}^{1}(\mathbb{R}^{n})$ to the subsequent transmission problem involving a \emph{conductive transmission condition}:
\begin{equation}
\left\{\begin{aligned}
&(\mL + \kappa^{2}\rho(x))w = h && \text{in $\Omega$,} \\
&(\Delta + \kappa^{2})w = 0 && \text{in $\mR^{n}\setminus\overline{\Omega}$,} \\
&(\nu \cdot A \nabla w)_{\rm int} - (\partial_{\nu} w)_{\rm ext} + \bfi \lambda w = g && \text{on $\partial\Omega$,} 
\end{aligned}\right. \label{eq:transmission-problem}
\end{equation}
where $\nu$ is the \emph{inward} unit normal vector to $\partial\Omega$ and formally we denote 
\begin{equation*}
\begin{aligned}
(\nu \cdot A \nabla w)_{\rm int}(x) &= \lim_{h\rightarrow 0_{+}} \nu(x) \cdot A(x + h\nu(x)) \nabla w(x+h\nu(x)) , \\
(\partial_{\nu} w)_{\rm ext}(x) &= \lim_{h\rightarrow 0_{+}} \nu(x) \cdot \nabla w(x-h\nu(x)) ,
\end{aligned} 
\end{equation*}
for a.e.\ $x\in\partial\Omega$.

When $A \equiv \Id$, the transmission problem \eqref{eq:transmission-problem} is associated with the interaction of a time-harmonic electromagnetic wave with an impermeable non-uniform structure encased by a thin, strongly conductive shell. This occurs under the conditions where the incoming electric field adheres to the transverse magnetic mode (TM-mode), and the derivation for this can be found in \cite[Section~1.2.1]{Bondarenko16thesis}.

A major difference between the above problem and the two-phase Bernoulli free boundary is the possibility of sign-change of solution in \eqref{eq:transmission-problem} within both $\Omega$ and its complement.
In the case of $g  > 0$ close to a boundary point $x_0 \in \partial \Omega$  one can actually show that the function $w$ does not change sign within each component $\Omega$, and $\mR^n \setminus \Omega$. 
This is a deep result in free boundary theory, and uses stronger form of monotonicity lemma (see Lemma~\ref{LEM:monotonicity-lemma} below) for more than two subharmonic functions, see e.g. \cite[Section~7]{ASP17MultiPlaseQD}, \cite[Theorem~3.1]{BFG21MultiPhaseMonotonicityLemma}, \cite[Lemmas~1.2 and 1.3]{CTV05MultiPlase} and \cite[Theorem~1.3]{Velichkov14MultiPlaseQD}. 
We refrain ourselves entering to the discussion here, but hope to get back to this in near future.

\medskip

Finally, we provide some observations regarding the elasticity system. Before introducing the elasticity tensor, let us introduce the notation 
\begin{equation*}
(\mathcal{A}:\mathcal{B})_{ijk\ell} = \sum_{p,q=1}^{n} A_{ijpq} B_{pqk\ell} \quad \text{for two tensors $\mathcal{A}$ and $\mathcal{B}$.}
\end{equation*}
Given an elasticity tensor $\mathcal{C} = (C_{ijk\ell})_{1 \le i,j,k,\ell \le n}$, adhering to major and minor symmetry, the behavior of elastic waves can be described by the equation 
\begin{equation*}
c(x)^{-2} \partial_{t}^{2} \vec{U} - \nabla \cdot (\mathcal{C} : (\nabla \otimes \vec{U})) = 0 \quad \left( \text{i.e. } c(x)^{-2} \partial_{t}^{2} \vec{U}_{i} - \sum_{j,k,\ell} \partial_{j} C_{ijk\ell}(x) \partial_{k} \vec{U}_{\ell} = 0 \right)
\end{equation*}
for the vector-valued function $\vec{U}$. Similarly, for a fixed constant $\kappa > 0$, one can analyze the scattering (with Kupradze radiation condition, see \cite{KW21CharacterizeNonradiating} or the monograph \cite{KGBB79elastic}) of elastic waves, corresponding to solutions of the form $\vec{U}(x,t) = e^{\bfi\kappa t} \vec{u}^{\rm to}$ that satisfy 
\begin{equation*}
\mL^{\mathcal{C}(x)} \vec{u} + \kappa^{2}\rho(x)\vec{u} = 0 \quad \text{in $\mathbb{R}^{n}$}
\end{equation*}
with $\rho = c^{-2}$ and $\mL^{\mathcal{C}(x)} \vec{u} = \nabla \cdot (\mathcal{C}(x) : (\nabla \otimes \vec{u}))$. 
It is worth noting that while the unique continuation property for the general elasticity system remains elusive, this property does hold true when the elasticity tensor is isotropic and adopts the form 
\begin{equation*}
C_{ijk\ell}(x) = C_{ijk\ell}^{\lambda(x),\mu(x)}(x) = \lambda(x)\delta_{ij}\delta_{k\ell} + \mu(x) (\delta_{ik}\delta_{j\ell} + \delta_{i\ell}\delta_{jk}).
\end{equation*}
When $\lambda$ and $\mu$ are constants (in this case they called the Lam\'{e} parameters), one also can write $\mL^{\lambda,\mu}\vec{u} \equiv \mL^{\mathcal{C}} = \mu \Delta \vec{u} + (\lambda+\mu)\nabla({\rm div}\,\vec{u})$.

As in Section \ref{sec_main}, one can investigate an elastic non-scattering problem similar to \eqref{eq:ITEP}:
\begin{equation*}
\begin{cases}
\left( \mL^{\mathcal{C}(x)} + \kappa^{2}\rho(x) \right) \vec{u}^{\rm to} = 0, \quad (\mL^{\lambda,\mu} + \kappa^{2})\vec{u}^{\rm inc} = 0, & \text{in $\Omega$,} \\
\vec{u}^{\rm to} = \vec{u}^{\rm inc} , \quad \nu \cdot \mathcal{C}(x) : (\nabla \otimes \vec{u}^{\rm to}) = \nu \cdot \mathcal{C}^{\lambda,\mu} : (\nabla \otimes \vec{u}^{\rm inc}), & \text{on $\partial \Omega$,}
\end{cases} 
\end{equation*}
as well as an analogue of the transmission problem \eqref{eq:transmission-problem}: 
\begin{equation*}
\left\{\begin{aligned}
&(\mL^{\mathcal{C}(x)} + \kappa^{2}\rho(x))\vec{w} = \vec{h} && \text{in $\Omega$,} \\
&(\mL^{\lambda,\mu} + \kappa^{2})\vec{w} = 0 && \text{in $\mR^{n}\setminus\overline{\Omega}$,} \\
&(\nu \cdot \mathcal{C}(x) : (\nabla \otimes \vec{w}))_{\rm int} - (\nu \cdot \mathcal{C}^{\lambda,\mu} : (\nabla \otimes \vec{w}))_{\rm ext} + \bfi \lambda \vec{w} = \vec{g} && \text{on $\partial\Omega$.} 
\end{aligned}\right. 
\end{equation*}
Here $\nu$ is the \emph{inward} unit normal vector to $\partial\Omega$ and formally we denote the inner and exterior traction operators by 
\begin{equation*}
\begin{aligned}
(\nu \cdot \mathcal{C}(x) : (\nabla \otimes \vec{w}))_{\rm int}(x) &= \lim_{h\rightarrow 0_{+}} \nu(x) \cdot \mathcal{C}(x+h\nu(x)) : (\nabla \otimes \vec{w})(x+h\nu(x)), \\
(\nu \cdot \mathcal{C}^{\lambda,\mu} : (\nabla \otimes \vec{w}))_{\rm int}(x) &= \lim_{h\rightarrow 0_{+}} \nu(x) \cdot \mathcal{C}(x-h\nu(x)) : (\nabla \otimes \vec{w})(x-h\nu(x)),
\end{aligned} 
\end{equation*}
for a.e.\ $x\in\partial\Omega$.
Here, we remind that the traction operator 
\begin{equation*}
\nu \cdot \mathcal{C}(x) : (\nabla \otimes \vec{u}^{\rm inc}) \quad \text{on $\partial\Omega$}
\end{equation*}
is a vector-valued function. Consequently, extending Theorem~\ref{THM:main1} to encompass the realm of elastic waves would require free boundary techniques for strongly coupled systems, which is currently out of reach.

\section{Proof of Theorem~{\rm \ref{THM:main1}} and Theorem~{\rm \ref{THM:non-radiating-source1}}}\label{Sec:2}

For many of the arguments below, we will follow \cite{ACS01FreeBoundaryCalderon}. For each $\epsilon>0$ and $L>0$, we define $Q_{\epsilon} \equiv Q_{\epsilon,L} := \{\abs{x'}<\epsilon\} \times (-2\epsilon L,2\epsilon L)$ and consider the graph 
\begin{equation*}
\Gamma_{\epsilon} \equiv \Gamma_{\epsilon,f} := \left\{ \begin{array}{c|c}(x',x_{n})\in\mR^{n} \times \mR & x_{n}=f(x') \text{ with } \abs{x'}<\epsilon \end{array} \right\}
\end{equation*}
of a Lipschitz function $f$, with $f(0)=0$ and Lipschitz constant $L$. Accordingly, we also define 
\begin{equation*}
\Lambda_{\epsilon} \equiv \Lambda_{\epsilon,f} = Q_{\epsilon,L} \cap \{x_{n} \le f(x')\} ,\quad \Omega_{\epsilon} \equiv \Omega_{\epsilon,f} := Q_{\epsilon,L}\setminus \Lambda_{\epsilon,f}.
\end{equation*}
Let $\nu$ be the unit normal vector to $\Gamma_{1}$ pointing towards the interior of $\Omega_{1}$. We denote $\mH^{n-1}\lfloor\Gamma_{\epsilon}$ the $(n-1)$-dimensional Hausdorff measure on $\Gamma_{\epsilon}$ and denote $\mathscr{L}^{n} \lfloor \Omega_{\epsilon}$ the Lebesgue measure on $\Omega_{\epsilon}$. 

The subsequent lemma can be derived using the exact methodology as outlined in \cite[Lemma~2.1]{ACS01FreeBoundaryCalderon}; the only changes in the proof are that one uses the properties of the fundamental solution given in \eqref{EQ:asymptotic-fundamental-solution} below and observes that $\|w\|_{L^{\infty}(Q_{3/2})}$ can be estimated by $\|w\|_{L^{2}(Q_{2})}$ by an interior elliptic regularity estimate.

\begin{lemma}\label{LEM:Holder-continuity}
Let $n \ge 2$ and $A \in (C^{\alpha}(Q_{2}))_{\rm sym}^{n \times n}$. If $w \in H^{1}(Q_{2})$ satisfies (in the sense of distribution) 
\begin{equation*}
\mL w = h \mathscr{L}^{n}\lfloor Q_{2} + g \mH^{n-1} \lfloor \Gamma_{2} \quad \text{in $Q_{2}$,}
\end{equation*}
for some $g \in L^{\infty}(\Gamma_{2})$ and $h \in L^{\infty}(Q_{2})$, then $w$ is H\"{o}lder continuous in $\overline{Q_{\frac{3}{2}}}$, and the H\"{o}lder constant depends only on $n,L,\alpha,\norm{w}_{L^{2}(Q_{2})}$, $\norm{h}_{L^{\infty}(Q_2)}$ and $\norm{g}_{L^{\infty}(\Gamma_{2})}$. 
\end{lemma}

Now the Lipschitz continuity of $w$ can be proved by slight modification of    ideas in \cite[Lemma~2.2]{ACS01FreeBoundaryCalderon}. 

\begin{lemma}\label{LEM:Lipschitz-continuity}
Let $n \ge 2$ and $A \in (C^{0,1}(Q_{2}))_{\rm sym}^{n \times n}$. 
If $w \in H^{1}(Q_{2})$ satisfies 
\begin{equation*}
\mL w = h \mathscr{L}^{n}\lfloor Q_{2} + g \mH^{n-1} \lfloor \Gamma_{2} \text{ in $Q_{2}$} ,\quad w=0 \text{ in $\Lambda_{2}$}
\end{equation*}
for some $g \in L^{\infty}(\Gamma_{2})$ and $h \in L^{\infty}(Q_{2})$, then $w$ is Lipschitz in $Q_{1}$, and the Lipschitz norm depends on $n,\alpha,\norm{w}_{L^{2}(Q_{2})}$, $\norm{h}_{L^{\infty}(Q_2)}$ and $\norm{g}_{L^{\infty}(\Gamma_{2})}$.
\end{lemma}

\begin{proof}
We only need to prove  the Lipschitz continuity of \(w\) at \(0 \in \Gamma_{1}\). In view of Lemma~{\rm \ref{LEM:Holder-continuity}}, without any compromise in generality, we can assume $\norm{w}_{L^{\infty}(Q_{\frac{3}{2}})}=1$.  Consequently, it suffices to establish the existence of a constants \(C, r_0\) such that \(\|w\|_{L^{\infty}(B_{r})} \le Cr\) for $r \leq r_0$.
We argue by contradiction and suppose that this fails. Then by using Lemma~\ref{LEM:Holder-continuity} there exists a sequence of continuous solutions $\{w_{j}\}$ and $r_{j}^{*} \searrow 0$ such that
\begin{equation*} 
\abs{w_{j}}\le 1 ,\quad \mL w_{j} = h_{j} \mathscr{L}^{n} \lfloor B_{3/2} + g_{j} \mH^{n-1} \lfloor \Gamma_{j} \text{ in $B_{3/2}$}, \quad w_{j}=0 \text{ in $\Lambda_{j}$,} 
\end{equation*}
where $\Gamma_{j}$ is a Lipschitz graph with Lipschitz constant $L$ and $0 \in \Gamma_{j}$, $\Lambda_{j}$ is the domain defined similar as above,  $\abs{g_{j}} \le \norm{g}_{L^{\infty}(\Gamma_{2})}$, $\abs{h_{j}} \le \norm{h}_{L^{\infty}(Q_{2})}$, satisfying 
\begin{equation}
\norm{w_{j}}_{L^{\infty}(B_{r_{j}^{*}})} \ge j r_{j}^{*}. \label{eq:contradiction}
\end{equation}
Since $\abs{w_{j}}\le 1$, from \eqref{eq:contradiction} one can easily see that $r_{j}^{*} \le j^{-1}$. By using $\abs{w_{j}}\le 1$ and the continuity of $w_{j}$, one can choose the largest $r_{j} \le j^{-1}$ such that the equality in \eqref{eq:contradiction} holds, that is,  
\begin{equation*} 
\norm{w_{j}}_{L^{\infty}(B_{r_{j}})} = j r_{j} \quad \text{and} \quad \norm{w_{j}}_{L^{\infty}(B_{r})} \le jr \text{ for all $r \ge r_{j}$.} \label{eq:contradiction1}
\end{equation*}  
If we define 
\begin{equation*}
\tilde{w}_{j}(x) = 
\frac{w_j(r_j x) }{j r_j} , 
\end{equation*}
then (for $j$ large enough such that $r_j < 1/10$ say) we have 
\begin{equation}
\norm{\tilde{w}_{j}}_{L^{\infty}(B_{2})} \le 2 ,\quad \norm{\tilde{w}_{j}}_{L^{\infty}(B_{1})} = 1. \label{EQ:to-be-contradict2}
\end{equation}
On the other hand, 
\begin{equation*}
\nabla_{x} \cdot \left( A(r_{j} x) \nabla \tilde{w}_{j}(x) \right) = \frac{r_{j}}{j} h_{j}(r_{j}x) \mathscr{L}^{n}\lfloor B_{2} + \frac{1}{j} g_{j}(r_{j}x) \mH^{n-1} \lfloor \Gamma_{j} \quad \text{in $B_{2}$.}
\end{equation*}
According to Lemma~{\rm \ref{LEM:Holder-continuity}}, we find that $\|\tilde{w}_{j}\|_{C^{\alpha}( B_{3/2})} \le C$, implying that $\{\tilde{w}_{j}\}$ is equicontinuous. Thus, by the Arzel{\`{a}}-Ascoli Theorem, a subsequence of $\{\tilde{w}_{j}\}$ converges uniformly within $B_{1}$ to $w_{\infty}$ satisfying a constant coefficient elliptic partial differential equation since for each $\varphi \in C_{c}^{\infty}(B_{1})$ one has 
\begin{equation*}
\begin{aligned}
& - \int_{B_{1}} (A(r_{j}x)-A(0)) \nabla \tilde{w}_{j}(x) \cdot \nabla \varphi(x) \, \rmd x \\
& \quad = \int_{B_{1}} (A(r_{j}x)-A(0)) \tilde{w}_{j}(x) \Delta \varphi(x) \, \rmd x + r_{j} \int_{B_{1}} \tilde{w}_{j}(x) (\nabla \cdot A)(r_{j}x) \cdot \nabla \varphi(x) \, \rmd x 
\end{aligned}
\end{equation*}
and the Lipschitz continuity of $A$ guarantees that the above equation converges to zero. 
Due to Lipschitz character of  $ \Gamma_1 $, 
the set $\{w_{\infty}=0\}$ possesses a non-empty interior, and consequently the unique continuation property ensures $w_{\infty}=0$ within $B_{3/2}$, contradicting \eqref{EQ:to-be-contradict2}.
\end{proof}

For readers' convenience, here we recall the mean value theorem for divergence form \cite{BH15EllipticDivergenceMVT,CR07ObstacleProblem}. See also the lecture note \cite{Caffarelli98ObstacleProblem} for a nice sketch of the ideas. 

\begin{lemma}\label{LEM:MVT}
Fix $n\ge 2$ and let $U\subset\mR^{n}$ be a bounded open set. Let $A \in (L^{\infty}(U))_{\rm sym}^{n\times n}$ be the real-valued symmetric matrix satisfies the uniform ellipticity condition \eqref{eq:unifurm-ellipticity}. For any $x_{0}\in U$ and
\begin{equation*}
v \in L^{1}(U) ,\quad \mL v \equiv \nabla\cdot A(x) \nabla v \ge 0 \text{ in $U$,}
\end{equation*}
there exist $c_{1}=c_{1}(n,c_{\rm ellip})>0$, $c_{2}=c_{2}(n,c_{\rm ellip})>0$ an increasing family of Borel sets $D_{R}^{\rm MVT}(x_{0})$ with 
\begin{equation}
B_{c_{1}R}(x_{0}) \subset D_{R}^{\rm MVT}(x_{0}) \subset B_{c_{2}R}(x_{0}) \subset U \label{EQ:compatible-ball}
\end{equation}
such that 
\begin{equation*}
v(x_{0}) \le \frac{1}{\abs{D_{R}^{\rm MVT}(x_{0})}} \int_{D_{R}^{\rm MVT}(x_{0})} v(x) \, \rmd x.
\end{equation*}
In addition, the mapping $R \mapsto \displaystyle{\frac{1}{\abs{D_{R}^{\rm MVT}(x_{0})}} \int_{D_{R}^{\rm MVT}(x_{0})} v(x) \, \rmd x}$ is monotone non-decreasing. Here $v(x_{0})$ is well-defined in the sense of its semicontinuous representative. 
\end{lemma}

Let $n \ge 3$ and $x_{0}\in Q_{\frac{1}{2},L}$. For convenience, later we will denote $D_{R}^{\rm MVT} = D_{R}^{\rm MVT}(0)$. One can fix $r_{0} = r_{0}(n,L)>0$ (independent of $x_{0}$) such that $\overline{B_{r_{0}}(x_{0})} \subset Q_{1,L}$.
For each $0 < r \le \frac{1}{2c_{2}}r_{0}$, 
let $\psi_{r}(\cdot) \in C_{c}^{1,1}
(\mR^{n}\setminus\{0 \})$ be the function 
satisfying 
$ \psi_{r} \ge \psi_{\delta r} $ for 
all $0<\delta<1$ and 
\begin{equation*}
\mL \psi_{r} = \frac{1}{\abs{D_{r}^{\rm MVT}}}\chi_{D_{r}^{\rm MVT}} - \delta_{0} \text{ in $\mD'(\mR^{n})$};
\end{equation*}
 see  Lemma~{\rm \ref{LEM:MVT}} in \cite{BH15EllipticDivergenceMVT} for a detailed proof of the existence of such function. We now define $\psi_{r,\delta} := \psi_{r} - \psi_{\delta r}$, which is a non-negative function satisfying 
\begin{equation}
\mL \psi_{r,\delta} = \frac{1}{\abs{D_{r}^{\rm MVT}}}\chi_{D_{r}^{\rm MVT}} - \frac{1}{\abs{D_{\delta r}^{\rm MVT}}}\chi_{D_{\delta r}^{\rm MVT}} \quad \text{in $\mD'(\mR^{n})$} \label{EQ:auxiliary-function}
\end{equation}
and 
\begin{equation}
\psi_{r,\delta} = 0 \text{ in $\mR^{n} \setminus D_{r}^{\rm MVT}$,} \label{EQ:support-condition}
\end{equation}
see also \cite[Figure~2]{Caffarelli98ObstacleProblem} for a graphical sketch of the ideas. Let $\Phi_{\mL}(x) = \Phi_{\mL}(x,0)$ be the Green's function of $\mL$ in $B_{r_{0}}$ \cite{GW82GreenFunction,LSW63EllipticDiscontinuousFundamentalSoln} in the sense of 
\begin{equation*}
\mL \Phi_{\mL} = - \delta_{0} \text{ in $B_{r_{0}}$} ,\quad \Phi_{\mL}=0 \text{ on $\partial B_{r_{0}}$,}
\end{equation*}
see also \cite{DHM18FundamentalSolution}. If $A$ is $C^{\alpha}$ near $x_{0}$ and $A(0)=\Id$, as in \cite[Lemma~1]{Caffarelli1988FreeBoundaryIII}, we have the asymptotic  behavior 
\begin{equation}
\begin{aligned}
\Phi_{\mL}(x) &= C_{0}\abs{x}^{2-n} + O(\abs{x}^{2-n+\alpha}), \\
\nabla \Phi_{\mL}(x) &= (2-n) C_{0} \abs{x}^{1-n}\hat{x} + O(\abs{x}^{1-n+\alpha}), 
\end{aligned}\label{EQ:asymptotic-fundamental-solution}
\end{equation}
where $\hat{x} = x/\abs{x}$ for $x \neq 0$. 
From this, we also know that there exists a constant $c = c(n,\alpha,L,c_{\rm ellip})>0$ (independent of $r$) such that 
\begin{equation}
c \abs{x}^{2-n} \le \lim_{\delta \rightarrow 0} 
\psi_{r,\delta}(x) \le c^{-1} \abs{x}^{2-n} 
\quad \text{for all 
 $\{ x \in B_{r_{0} } \setminus  0 \}$ .} 
\label{EQ:singularity}
\end{equation}
We are now able to prove the following non-degeneracy result by modifying the proof of Lemma~{\rm \ref{LEM:MVT}}, see also \cite[Lemma~3.1]{ACS01FreeBoundaryCalderon}. 

\begin{lemma}\label{LEM:non-degeneracy-result}
Let $n\ge 3$, let $x_{0} \in \Gamma_{\frac{1}{2}}$, let $A \in (C^{\alpha}(Q_{1}))_{\rm sym}^{n\times n}$, and let $w \in C^{0,1}(Q_{1})$ satisfy 
\begin{equation*}
\mL w \ge c_{3} \mH^{n-1} \lfloor \Gamma_{1} - c_{3}' \mathscr{L}^{n} \lfloor Q_{1} \text{ in $Q_{1}$} ,\quad w(x_{0}) \ge 0,
\end{equation*}
for some constants $c_{3},c_{3}'>0$. Then there exist positive constants $c = c(n,f,c_{\rm ellip},c_{3})$ and $r_{1} = r_{1}(n,L,c_{\rm ellip})$ such that 
\begin{equation*}
\frac{1}{\abs{D_{r}^{\rm MVT}(x_{0})}} \int_{D_{r}^{\rm MVT}(x_{0})} w(x) \, \rmd x \ge c r \quad \text{for all $0 < r \le r_{1}$,}
\end{equation*}
where $D_{r}^{\rm MVT}(x_{0})$ is the set appearing  in the mean value theorem {\rm (}Lemma~{\rm \ref{LEM:MVT}}{\rm )}. 
\end{lemma}

\begin{remark*} 
By using \eqref{EQ:compatible-ball}, one sees that the positive part $w_{+} = \max \{w,0\}$ satisfies 
\begin{equation*}
\begin{aligned}
& \frac{1}{\abs{B_{r}(x_{0})}} \int_{B_{r}(x_{0})} \abs{w_{+}(x)}^{2} \, \rmd x \ge \left( \frac{1}{\abs{B_{r}(x_{0})}} \int_{B_{r}(x_{0})} w_{+}(x) \, \rmd x \right)^{2} \\
& \quad \ge \left( \frac{D_{c_{2}^{-1}r}^{\rm MVT}(x_{0})}{\abs{B_{r}(x_{0})}} \frac{1}{D_{c_{2}^{-1}r}^{\rm MVT}(x_{0})} \int_{D_{c_{2}^{-1}r}^{\rm MVT}(x_{0})} w_{+}(x) \, \rmd x \right)^{2} \\
& \quad \ge \left( \frac{B_{c_{1}c_{2}^{-1}r}(x_{0})}{B_{r}(x_{0})} c r \right)^{2} = c' r^{2}
\end{aligned}
\end{equation*}
for all $0 < r \le r_{2}$ with $r_{2} = c_{2}^{-1}r_{1}$, where $c' = c'(n,f,c_{\rm ellip},c_{3}) > 0$. For later purpose, here we also denote the negative part $w_{-} := -\min\{w,0\} = \max \{-w,0\}$ of $w$. 
\end{remark*}

\begin{proof}[Proof of Lemma~{\rm \ref{LEM:non-degeneracy-result}}]
By a linear change of variables, it  suffices  to prove the result when $A(x_{0})=\Id$ and $x_{0}=0$. Since $w(0) \ge 0$, by continuity of $w$ we have 
\begin{equation*}
\liminf_{\delta \rightarrow 0} \frac{1}{\abs{D_{\delta r}^{\rm MVT}}} \int_{D_{\delta r}^{\rm MVT}} w(x) \, \rmd x \ge 0.
\end{equation*}
From \eqref{EQ:support-condition} and \eqref{EQ:auxiliary-function} we have 
\begin{equation*}
\begin{aligned}
& \frac{1}{\abs{D_{r}^{\rm MVT}}}\int_{D_{r}^{\rm MVT}}w(x) \, \rmd x - \frac{1}{\abs{D_{\delta r}^{\rm MVT}}} \int_{D_{\delta r}^{\rm MVT}} w(x) \, \rmd x = \int_{B_{r_{0}}} w(x) \mL \psi_{r,\delta}(x) \, \rmd x \\
& \quad \ge c_{3} \int_{\Gamma_{1}\cap B_{c_{1}r}} \psi_{r,\delta}(x) \, \rmd \mH^{n-1} - c_{3}' \int_{B_{c_{1}r}} \psi_{r,\delta}(x) \, \rmd x
\end{aligned}
\end{equation*}
where $c_{1} = c_{1}(n,c_{\rm ellip})>0$ is the constant appearing  in Lemma~{\rm \ref{LEM:MVT}}. Letting  $\delta \rightarrow 0$, the claim in the  lemma follows from \eqref{EQ:singularity}. 
\end{proof}

Our goal now is to demonstrate the non-negativity of \(w\). To achieve this, we require a H\"older   upper bound  for the so-called  monotonicity function   
(see  \cite[Lemma~5.1] {ACF84TwoPhasesFreeBoundary}).  Related discussions concerning the monotonicity function (also called monotonicity formula) can  also be  found in \cite{ACS01FreeBoundaryCalderon,CKS00FreeBoundaryPompeiu, CK98GradientEstimates,MP11AlmostMonotonicity}.

\begin{lemma} \label{LEM:monotonicity-lemma}
Let $n \ge 3$ and let $x_{0} 
\in \Gamma_{\frac{1}{2}}$. Let $A \in (C^{\alpha}(Q_{1}))_{\rm sym}^{n \times n}$ for some $0<\alpha<1$, and let $w \in C^{0,1}(Q_{1})$ satisfy 
\begin{equation*}
\mL w = h \text{ in $\Omega_{1}$} ,\quad w=0 \text{ in $\Lambda_{1}$}
\end{equation*}
for some $h \in L^{\infty}(Q_{1})$. Then there exists constants $\theta > 0$ and $C>0$ such that 
\begin{equation}
\frac{1}{r^{4}} \left( \int_{B_{r}(x_{0})} \frac{\abs{\nabla w_{+}}^{2}}{\abs{x}^{n-2}} \, \rmd x \right) \left( \int_{B_{r}(x_{0})} \frac{\abs{\nabla w_{-}}^{2}}{\abs{x}^{n-2}} \, \rmd x \right) \le C r^{\theta} \label{eq:refine-monotonicity}
\end{equation}
for all sufficiently small $r>0$. 
\end{lemma}

The above lemma can be validated through an approach closely aligned with the principles outlined in the free boundary literature \cite[Lemma~5.1]{ACF84TwoPhasesFreeBoundary} and \cite[Lemma~1]{Caffarelli1988FreeBoundaryIII}. For sake of completeness and being self-contained, we present the detailed proof in Appendix~\ref{appen:monotonicity-lemma}.

We now prove a lemma, which is analogue to \cite[Lemma~3.2]{ACS01FreeBoundaryCalderon}.

\begin{lemma} \label{LEM:negative-part-decay}
Let $n \ge 2$ and let $x_{0} 
\in \Gamma_{\frac{1}{4}}$. Let $A \in (C^{\alpha}(Q_{1}))_{\rm sym}^{n \times n}$ for some $0<\alpha<1$, and let $w \in C^{0,1}(Q_{1})$ satisfy 
\begin{equation*}
\mL w = h \text{ in $\Omega_{1}$} ,\quad \mL w \ge h \mathscr{L}^{n}\lfloor\Omega_{1} + c_{3} \mH^{n-1}\lfloor \Gamma_{1} \text{ in $Q_{1}$} ,\quad w=0 \text{ in $\Lambda_{1}$}
\end{equation*}
for some $c_{3} > 0$ and $h \in L^{\infty}(Q_{1})$, then 
\begin{equation*}
\lim_{r \rightarrow 0} \frac{1}{r} \sup_{B_{r}(x_{0})} w_{-} = 0.
\end{equation*}
\end{lemma}

\begin{proof} 
It  suffices to prove the lemma for $n \ge 3$. For $n = 2$ one can add $x_3$ as a dummy variable to $w$, $A$ and $h$, extend $A$ as a symmetric matrix with $a_{j3}(x) = \delta_{j3}$, and extend $\Omega_1$, $Q_1$, $\Gamma_1$, and $\Lambda_1$ as constant in the $x_3$ direction. The result for $n=2$ then reduces to the case $n=3$. 

By using Poincar\'{e} inequality and Lemma~{\rm \ref{LEM:monotonicity-lemma}}, one sees that 
\begin{equation*}
\begin{aligned}
& \frac{1}{r^{4}} \left( \frac{1}{B_{r}(x_{0})} \int_{B_{r}(x_{0})} w_{+}^{2} \, \rmd x \right) \left( \frac{1}{B_{r}(x_{0})} \int_{B_{r}(x_{0})} w_{-}^{2} \, \rmd x \right) \\
& \quad \le C \frac{1}{r^{4}} \left( \int_{B_{r}(x_{0})} \abs{\nabla w_{+}}^{2} \, \rmd x \right) \left( \int_{B_{r}(x_{0})} \abs{\nabla w_{-}}^{2} \, \rmd x \right) \le C r^{\theta}
\end{aligned}
\end{equation*}
for all sufficiently small $r>0$. Combining the above inequality with Lemma~{\rm \ref{LEM:non-degeneracy-result}}, we reach 
\begin{equation*}
\frac{1}{r^{2}} \left( \frac{1}{B_{r}(x_{0})} \int_{B_{r}(x_{0})} w_{-}^{2} \, \rmd x \right) \le Cr^{\theta}
\end{equation*}
Finally, arguing as in \cite[Lemma~3.2]{ACS01FreeBoundaryCalderon}, we conclude our lemma. 
\end{proof}

We now show the positvity of $w$ near the free boundary $\Gamma_{\frac{1}{4}}$.

\begin{lemma} \label{LEM:positive-near-boundary}
Suppose that all assumptions in Lemma~{\rm \ref{LEM:Lipschitz-continuity}} hold. Then,  for each $x_{0} \in \Gamma_{\frac{1}{4}}$,  there exists $\delta > 0$, which is independent of $x_{0}$, such that $w(x) > 0$ in $B_{\delta}(x_{0})$. \end{lemma}

\begin{remark*}
The case when $h \le 0$ also can be proved by following closely to the quantitative arguments in \cite[Lemma~3.3]{ACS01FreeBoundaryCalderon}, by using the doubling property of $\mathcal{L}$-harmonic measure (also known as elliptic measure) given in \cite{CFMS81HarmonicMeasure}, see also the monographs \cite{CS05FreeBoundary,Kenig94HarmonicAnalysis} or \cite{LP19HarmonicMeasure}. 
\end{remark*}

\begin{proof}[Proof of Lemma~{\rm \ref{LEM:positive-near-boundary}}]
Suppose the contrary that such a $\delta$ does not exist. Then there exists a sequence $\{x_{j}\}_{j\in\mathbb{N}} \subset \Omega_{1}$ such that $w(x_{j})=0$ and $x_{j} \rightarrow x_{0}$ for some $x_{0} \in \Gamma_{\frac{1}{4}}$. Let 
\begin{equation*}
w_{j}(x) := \frac{w(d_{j}x + x_{j})}{d_{j}} \text{ for all $x \in B_{1}$} ,\quad d_{j} := \dist \, (x_{j}, \Gamma_{\frac{1}{4}}).
\end{equation*}
Since $x \in B_{1} \iff d_{j}x + x_{j} \in B_{d_{j}}(x_{j}) \subset \Omega_{1}$, then 
\begin{equation*}
\nabla \cdot ( A(d_{j}x + x_{j}) \nabla w_{j} ) = d_{j} h (d_{j}x + x_{j}) \quad \text{for all $x \in B_{1}$.}
\end{equation*}

Similar to Lemma~{\rm \ref{LEM:Lipschitz-continuity}}, through the utilization of the Arzel\`{a}-Ascoli theorem, we identify a subsequence -- still denoted as $\{w_{j}\}$ -- that uniformly converges in $B_{\frac{1}{2}}$ to a function $w_{\infty}$, satisfying the equation
\begin{equation*}
\nabla \cdot A(0) \nabla w_{\infty} = 0 \quad \text{in $B_{\frac{1}{2}}$.}
\end{equation*}
By applying Lemma~{\rm \ref{LEM:negative-part-decay}}, it becomes evident that $w_{\infty} \ge 0$ in $B_{\frac{1}{2}}$. Furthermore, invoking Lemma~{\rm \ref{LEM:non-degeneracy-result}}, we deduce that $w_{\infty} \not\equiv 0$ in $B_{\frac{1}{2}}$. Consequently, the strong minimum principle \cite[Theorem~8.19]{GT01Elliptic} implies that $w > 0$ in $B_{\frac{1}{2}}$, which contradicts the fact that $w_{\infty}(0)=0$. Thus, our lemma is conclusively established.
\end{proof}

We are now ready to state and proof the main result of this section, akin to what is presented in \cite[Section~4]{ACS01FreeBoundaryCalderon}.

\begin{proposition}\label{PROP:free-boundary1}
Let $n \ge 2$, let $A \in (C^{1}(Q_{2}))_{\rm sym}^{n \times n}$, and let $w \in H^{1}(Q_{2})$ satisfy 
\begin{equation*}
\mL w = h \mathscr{L}^{n}\lfloor \Omega_{2} + (\nu \cdot A \bfV)  \mH^{n-1} \lfloor \Gamma_{2} \text{ in $Q_{2}$} ,\quad w =0 \text{ in $\Lambda_{2}$} 
\end{equation*}
for $h \in L^{\infty}(\Omega_{2})$ and some Lipschitz continuous vector field $\bfV$ with
\begin{equation*}
\nu \cdot A \bfV \ge c_{3} \text{ on $\Gamma_{2}$}
\end{equation*}
for some $c_{3}>0$. Then $\Gamma_{\frac{1}{4}}$ is $C^{1,\alpha'}$. 
\end{proposition}

\begin{proof}
It suffices to demonstrate that $\Gamma_{1}$ exhibits $C^{1,\alpha'}$ regularity near 0 for dimensions $n \ge 3$.
We begin by proving that $w$ is a viscosity solution. To show this, following \cite{ACS01FreeBoundaryCalderon}, we begin by looking at those points $x_{0} \in \Gamma_{\frac{1}{4}}$ for which there exists a ball $B$ contained within $\Omega_{\frac{1}{4}}$ touching $\Gamma_{1}$ at $x_{0}$. Consider such a point $x_{0}$, and let $\nu_{0}$ denote the unit normal vector to $\partial B$ at $x_{0}$, directed towards the interior of $\Omega_{\frac{1}{4}}$. Analogous to \cite[(4.1)]{ACS01FreeBoundaryCalderon}, utilizing Lemma~{\rm \ref{LEM:non-degeneracy-result}} and Lemma~{\rm \ref{LEM:positive-near-boundary}} (also see \cite[Lemma~11.17]{CS05FreeBoundary}), there exists a positive $\beta$ such that
\begin{equation}
w(x) = \beta \left( (x - x_{0})\cdot \nu_{0} \right)_{+} + o (\abs{x-x_{0}}). \label{eq:to-ask}
\end{equation}
Subsequently, we assert that
\begin{equation}
\beta = \nu_{0} \cdot A(x_{0}) \bfV(x_{0}). \label{EQ:to-do}
\end{equation}
Once this verification is achieved, it becomes evident that $w$ is   a  viscosity solution (in the context of \cite{DSFS14TwoPhaseProblemLinear}, as also discussed in the monograph \cite{CS05FreeBoundary}) for the ensuing one-phase problem within $B_{\delta}$, where $\delta>0$ is small, as provided by Lemma~{\rm \ref{LEM:positive-near-boundary}}:
\begin{equation*} \begin{cases} \mL w = h & 
\text{in $B_{\delta} \cap \{w>0\}$,} \\ 
\partial_{\nu} w = \nu \cdot A \bfV & 
\text{on 
$B_{\delta} \cap \partial \{w>0\}$.} \end{cases} \end{equation*}
Finally, by employing the free boundary regularity outcome from \cite[Theorem~1.4]{DSFS14TwoPhaseProblemLinear}\footnote{This problem was initially explored in \cite{Caffarelli1987FreeBoundaryI}.
In an effort to encompass all pertinent existing findings, we refer to \cite{STV19TwoPhaseProblemFullyLinear} for fully nonlinear equations, as well as a comprehensive survey paper \cite{DSFS15TwoPhaseProblemSurvey} for further insights.}, we arrive at the fact that the free boundary $\Gamma_1$ is $C^{1,\alpha'}$.

To conclude, we are left with the task of verifying \eqref{EQ:to-do}. To streamline the discussion, we shall focus on establishing \eqref{EQ:to-do} under the assumptions of $A(x_{0})=\rm Id$, $x_{0}=0$, and $\nu_{0}=e_{1}$, which can be achieved by a linear change of variables, translation and rotation.

Consider $\tilde{\psi}_{2r}$ which satisfies
\begin{equation*}
\Delta \tilde{\psi}_{2r} = \frac{1}{\abs{B_{2r}}} \chi_{B_{2r}} - \frac{1}{\abs{B_{2r}}} \chi_{B_{r}} \quad \text{in $\mD'(\mR^{n})$},
\end{equation*}
as witnessed in the proof of \cite[Lemma~3.1]{ACS01FreeBoundaryCalderon}.
Through straightforward computations, it becomes apparent that
\begin{equation*}
\begin{aligned}
& \int_{\Gamma_{2r}} \tilde{\psi}_{2r} \nu \cdot \bfV \, \rmd \mH^{n-1} = \int_{B_{2r}} w \mL \tilde{\psi}_{2r} \, \rmd x \\
& \quad = \frac{1}{\abs{B_{2r}}} \int_{B_{2r}} w(x) \, \rmd x - \frac{1}{\abs{B_{r}}} \int_{B_{r}} w(x) \, \rmd x \\
& \qquad + \int_{B_{2r}} h \tilde{\psi}_{2r} \, \rmd x + \int_{B_{2r}} w \nabla \cdot (A - \Id) \nabla \tilde{\psi}_{2r} \, \rmd x .
\end{aligned}
\end{equation*}
Since $\norm{w}_{L^{\infty}(B_{r})} \le cr$, $\norm{A - \Id}_{L^{\infty}(B_{r})} \le cr$, $\norm{\nabla \tilde{\psi}_{2r}}_{L^{\infty}(B_{r})} \le c r^{1-n}$ and $\norm{\nabla^{2} \tilde{\psi}_{2r}}_{L^{\infty}(B_{r})} \le c r^{-n}$, 
then 
\begin{equation*}
\frac{1}{r} \int_{B_{2r}} w \nabla \cdot (A - \Id) \nabla \tilde{\psi}_{2r} \, \rmd x \rightarrow 0 \quad \text{as $r \rightarrow 0$.}
\end{equation*}
Since $\norm{\tilde{\psi}_{2r}}_{L^{\infty}(B_{r})} \le cr^{2-n}$, then 
\begin{equation*}
\frac{1}{r} \int_{B_{2r}} h \tilde{\psi}_{2r} \, \rmd x \rightarrow 0 \quad \text{as $r \rightarrow 0$.}
\end{equation*}
It was shown  in \cite[Lemma~4.1]{ACS01FreeBoundaryCalderon} that 
\begin{equation*}
\frac{1}{r} \left( \frac{1}{\abs{B_{2r}}} \int_{B_{2r}} w(x) \, \rmd x - \frac{1}{\abs{B_{r}}} \int_{B_{r}} w(x) \, \rmd x \right) \rightarrow \frac{v_{n-1}}{(n+1)v_{n}} \beta \quad \text{as $r\rightarrow 0$}, 
\end{equation*}
where $v_{n}$ is the volume of the unit ball in $\mR^{n}$. From \cite[(4.4)--(4.6)]{ACS01FreeBoundaryCalderon} and the Lipschitz continuity of $\bfV$, we also know that 
\begin{equation}
\frac{1}{r} \int_{\Gamma_{2r}} \tilde{\psi}_{2r} \nu \cdot \bfV \, \rmd \mH^{n-1} \rightarrow \frac{v_{n-1}}{(n+1)v_{n}} e_{n} \cdot \bfV(0) \quad \text{as $r\rightarrow 0$.} \label{eq:limit1}
\end{equation}
Combining the equations above, we conclude \eqref{EQ:to-do}. 
\end{proof}

The proofs of the main results of this paper now follow directly from the lemmas demonstrated earlier.

\begin{proof}[Proof of Theorem~{\rm \ref{THM:main1}}]
We only need to prove the theorem for the first case in \eqref{EQ:non-degeneracy-condition}. We define the function 
\begin{equation*}
h(x) = -\mL u^{\rm inc} - \kappa^{2}\rho u^{\rm to},
\end{equation*}
which is continuous up to $\partial\Omega$ near $x_{0}$. Our theorem immediately follows from Proposition~{\rm \ref{PROP:free-boundary1}} with $\bfV = A^{-1}(\Id - A) \nabla u^{\rm inc}$, which is Lipschitz in $\Omega$ since $A$ is Lipschitz in $\Omega$. 
\end{proof}

\begin{proof}[Proof of Theorem~{\rm \ref{THM:non-radiating-source1}}]
This is an immediate consequence of Proposition~{\rm \ref{PROP:free-boundary1}}. 
\end{proof}

\appendix

\section{Proof of almost monotonicity lemma\label{appen:monotonicity-lemma}}

The main theme of this appendix is to prove Lemma~\ref{LEM:monotonicity-lemma}. 

\begin{proof}[Proof of Lemma~\ref{LEM:monotonicity-lemma}]
By using \cite[Theorem~II.6.6]{KS00IntroductionVariationalInequalities}, one sees that $w_{\pm} \in C^{0,1}(Q_{1})$ and  
\begin{equation*}
\mL w_{\pm} \ge -M \text{ in $\Omega_{1}$} ,\quad w=0 \text{ in $\Lambda_{1}$.}
\end{equation*}
for some $M>0$. We only need to prove the result when 
there exists $r' > 0$ such that 
\begin{equation}
\max_{x \in \partial B_{r}(x_{0})} w_{-} > 0 \quad \text{for all $0<r<r'$.} \label{EQ:non-degeneracy-assumption}
\end{equation}
Otherwise, by using the maximum principle for elliptic equations \cite{GT01Elliptic}, one sees that $w_{-}=0$ near $x_{0}$, for which the result trivially holds. 

In order to deliver our ideas clearly, we divide the proof into several steps.

\medskip

\noindent \textbf{Step 1: A basic estimate.} 
One sees that 
\begin{equation*}
\begin{aligned}
& \mL \abs{w_{\pm}}^{2} = 2 \nabla \cdot (w_{\pm} A\nabla w_{\pm}) = 2 \nabla w_{\pm} \cdot A \nabla w_{\pm} + 2 w_{\pm} \mL w_{\pm} \\
& \quad \ge 2 \nabla w_{\pm} \cdot A \nabla w_{\pm} - 2M w_{\pm} \quad \text{in $Q_{1}$.} 
\end{aligned}
\end{equation*}
Without loss of generality, it suffices to  prove the lemma for $A(x_{0})=\Id$ and $x_{0}=0$. Let $\Phi_{\mL}$ be the Green's function of $-\mL$ as in \eqref{EQ:asymptotic-fundamental-solution}. Then for each $0 < \vareps < r$ one sees that 
\begin{equation}
\begin{aligned}
& 2 \int_{B_{r}\setminus \overline{B_{\epsilon}}} \nabla w_{\pm} \cdot A \nabla w_{\pm} \Phi_{\mL} \, \rmd x \le \int_{B_{r}\setminus \overline{B_{\epsilon}}} \mL  \abs{w_{\pm}}^{2} \Phi_{\mL} \, \rmd x + 2M \int_{B_{r}} w_{\pm} \Phi_{\mL} \, \rmd x \\
& \quad = \int_{\partial B_{r}} \hat{x} \cdot A \nabla (\abs{w_{\pm}}^{2}) \Phi_{\mL} \, \rmd \mH^{n-1} - \int_{\partial B_{r}}  \abs{w_{\pm}}^{2} \hat{x} \cdot A \nabla \Phi_{\mL} \, \rmd \mH^{n-1} - I_{\vareps} \\
& \qquad + 2M \int_{B_{r}} w_{\pm} \Phi_{\mL} \, \rmd x,
\end{aligned} \label{EQ:apriori-monotonicity1}
\end{equation}
where 
\begin{equation*}
I_{\vareps} = \int_{\partial B_{\vareps}} \hat{x} \cdot A \nabla (\abs{w_{\pm}}^{2}) \Phi_{\mL} \, \rmd \mH^{n-1} - \int_{\partial B_{\vareps}}  \abs{w_{\pm}}^{2} \hat{x} \cdot A \nabla \Phi_{\mL} \, \rmd \mH^{n-1} .
\end{equation*}
Since $\abs{\nabla w_{\pm}}$ is bounded, together with \eqref{EQ:asymptotic-fundamental-solution}, by computing as in the proof of \cite[page~439]{ACF84TwoPhasesFreeBoundary} one reach 
\begin{equation*}
\lim_{\vareps \rightarrow 0} I_{\vareps} = (n-2) \abs{\partial B_{1}} \abs{w_{\pm}(0)}^{2} \ge 0.
\end{equation*}
On the other hand, since $w_{\pm}$ is Lipschitz and $w_{\pm}(0)=0$, then from \eqref{EQ:asymptotic-fundamental-solution} we obtain 
\begin{equation*}
\begin{aligned}
& \int_{B_{r}} w_{\pm} \Phi_{\mL} \, \rmd x \le Cr \int_{B_{r}} \Phi_{\mL} \, \rmd x \\
& \quad \le Cr \int_{B_{r}} \abs{x}^{2-n} \, \rmd x + Cr \int_{B_{r}} \abs{x}^{2-n+\alpha} \, \rmd x \le Cr^{3}. 
\end{aligned}
\end{equation*}
Therefore \eqref{EQ:apriori-monotonicity1} implies 
\begin{equation}
\begin{aligned}
& 2 \ell_{\pm}(r) := 2 \int_{B_{r}} \nabla w_{\pm} \cdot A \nabla w_{\pm} \Phi_{\mL} \, \rmd x \\
&\quad \le \int_{\partial B_{r}} \hat{x} \cdot A \nabla (\abs{w_{\pm}}^{2}) \Phi_{\mL} \, \rmd \mH^{n-1} - \int_{\partial B_{r}}  \abs{w_{\pm}}^{2} \hat{x} \cdot A \nabla \Phi_{\mL} \, \rmd \mH^{n-1} + Cr^{3},
\end{aligned} \label{EQ:apriori-monotonicity2}
\end{equation}
which is a  crucial estimate in the rest of the proof to follow. 

\medskip

\noindent \textbf{Step 2: A surface eigenvalue problem.} In view of 
\eqref{EQ:non-degeneracy-assumption}, we now write $(\partial B_{r})_{\pm} := \{ w_{\pm} > 0 \} \cap \partial B_{r}$ and see that $(\partial B_{1})_{\pm} := r^{-1} (\partial B_{r})_{\pm} \subset \partial B_{1}$ as well as $\mH^{n-1}((\partial B_{1})_{\pm}) = r^{1-n} \mH^{n-1}((\partial B_{r})_{\pm}) > 0$. Since $w_{+} \cdot w_{-} = 0$, then 
\begin{equation*}
\mH^{n-1}((\partial B_{1})_{+}) + \mH^{n-1}((\partial B_{1})_{-}) \le \mH^{n-1}(\partial B_{1}).
\end{equation*}
Since $\Gamma_{1}$ is Lipschitz and $w=0$ in $\Lambda_{1}$, then $w_{\pm}$ vanishes in a cone, hence there exists $0 < \theta < \frac{1}{4}$ (say), which is independent of $x_{0}$, such that 
\begin{equation*}
s_{+} + s_{-} \le 1 - \theta ,\quad s_{\pm} := \frac{\mH^{n-1}((\partial B_{1})_{+})}{\mH^{n-1}(\partial B_{1})}. 
\end{equation*}
Let $\nabla_{\partial B_{1}}$ be the gradient of a function $v$ on $\partial B_{1}$. We introduce the constant $\alpha_{\pm}$ given by 
\begin{equation*}
\alpha_{\pm} := \inf_{v \in H_{0}^{1}((\partial B_{1})_{\pm})} \frac{\int_{(\partial B_{1})_{\pm}} \abs{\nabla_{\partial B_{1}}v}^{2} \, \rmd \mH^{n-1}}{\int_{(\partial B_{1})_{\pm}} \abs{v}^{2} \, \rmd \mH^{n-1}}. 
\end{equation*}
For each small $r>0$, we define $\tilde{w}_{\pm}(\hat{x}) := w_{\pm}(r\hat{x})$ for all $\hat{x} \in \partial B_{1}$. For any $0 < \beta_{\pm} < 1$, we can write 
\begin{equation*}
\begin{aligned}
& \int_{\partial B_{1}} \left( (\hat{x}\cdot\nabla \tilde{w}_{\pm})^{2} + \beta^{2} \abs{\nabla_{\partial B_{1}} \tilde{w}_{\pm}}^{2} \right) \, \rmd \mH^{n-1} \\
& \quad \ge 2 \left( \int_{\partial B_{1}} (\hat{x}\cdot\nabla \tilde{w}_{\pm})^{2} \, \rmd \mH^{n-1} \right)^{\frac{1}{2}} \left( \int_{\partial B_{1}} \beta^{2} \abs{\nabla_{\partial B_{1}} \tilde{w}_{\pm}}^{2} \, \rmd \mH^{n-1} \right)^{\frac{1}{2}} \\
& \quad \ge \frac{2 \beta_{\pm}}{\sqrt{\alpha_{\pm}}} \left( \int_{\partial B_{1}} (\hat{x}\cdot\nabla \tilde{w}_{\pm})^{2} \, \rmd \mH^{n-1} \right)^{\frac{1}{2}} \left( \int_{\partial B_{1}} \abs{\tilde{w}_{\pm}}^{2} \, \rmd \mH^{n-1} \right)^{\frac{1}{2}} \\
& \quad \ge \frac{2 \beta_{\pm}}{\sqrt{\alpha_{\pm}}} \int_{\partial B_{1}} \abs{\tilde{w}_{\pm} \hat{x}\cdot \nabla \tilde{w}_{\pm}} \, \rmd \mH^{n-1}
\end{aligned}
\end{equation*}
and 
\begin{equation*}
\int_{\partial B_{1}} (1 - \beta_{\pm}^{2})  \abs{\nabla_{\partial B_{1}} \tilde{w}_{\pm}}^{2} \, \rmd \mH^{n-1} \ge \frac{1 - \beta_{\pm}^{2}}{\alpha_{\pm}} \int_{\partial B_{1}} \tilde{w}_{\pm}^{2} \, \rmd \mH^{n-1}.
\end{equation*}
We now choose  
\begin{equation*}
\beta_{\pm} = \frac{\sqrt{\alpha_{\pm}}}{2} \left( \left( (n-2)^{2} + \frac{4}{\alpha_{\pm}} \right)^{\frac{1}{2}} - (n-2) \right) ,\quad \gamma_{\pm} = \frac{\beta_{\pm}}{\sqrt{\alpha_{\pm}}}. 
\end{equation*}
By direct computations, we see that 
\begin{equation*}
\frac{1 - \beta_{\pm}^{2}}{\alpha_{\pm}} = (n-2) \frac{\beta_{\pm}}{\sqrt{\alpha_{\pm}}} = (n-2) \gamma_{\pm}
\end{equation*}
and 
\begin{equation}
\int_{\partial B_{1}} \abs{\nabla \tilde{w}_{\pm}}^{2} \, \rmd \mH^{n-1} \ge \gamma_{\pm} \left( \int_{\partial B_{1}} 2 \abs{\tilde{w}_{\pm} \hat{x}\cdot\nabla \tilde{w}_{\pm}} \, \rmd \mH^{n-1} + (n-2) \int_{\partial B_{1}} \tilde{w}_{\pm}^{2} \, \rmd \mH^{n-1} \right). \label{EQ:surface-estimate1}
\end{equation}
By using \cite[Theorem~E, Theorem~2 and Theorem~3]{FH76EigenvalueSphere}\footnote{The fundamental result \cite[Theorem~E]{FH76EigenvalueSphere} was proved in \cite{Sperner73EigenvalueSphere}.}\footnote{See also \cite[Section~2.4]{CK98GradientEstimates} for some discussions on a convexity property of the first Dirichlet eigenvalue of the Orstein-Uhlembeck operator $\Delta - x \cdot \nabla$ on a (sufficiently regular) open set in $\mathbb{R}^{n}$.}, one has 
\begin{equation*}
\gamma_{\pm} \ge \varphi(s_{\pm}), \quad \varphi(s) = 
\begin{cases}
\displaystyle{\frac{1}{2}\log\frac{1}{4s} + \frac{3}{2}} & \text{if $s < \displaystyle{\frac{1}{4}}$} \\
2(1-s) & \text{if $\displaystyle{\frac{1}{4}} \le s < 1$.}
\end{cases}
\end{equation*}
Since $\varphi$ is convex, then  
\begin{equation}
\gamma_{+} + \gamma_{-} \ge \varphi(s_{+}) + \varphi(s_{-}) \ge 2 \varphi \left( \frac{s_{+} + s_{-}}{2} \right) \ge 2 \varphi \left( \frac{1 - \theta}{2} \right) = 2 + 2 \theta. \label{EQ:surface-estimate2}
\end{equation}
From \eqref{EQ:surface-estimate1}, we obtain 
\begin{equation*}
r \int_{\partial B_{r}} \abs{\nabla w_{\pm}}^{2} \, \rmd \mH^{n-1} \ge \gamma_{\pm} \left( \int_{\partial B_{r}}  \abs{\hat{x}\cdot\nabla (w_{\pm}^{2})} \, \rmd \mH^{n-1} + (n-2)r^{-1} \int_{\partial B_{r}} w_{\pm}^{2} \, \rmd \mH^{n-1} \right).
\end{equation*}
Since 
\begin{equation*}
\frac{1}{r^{2}} \ell_{\pm}(r) \le \frac{1}{r^{2}} \norm{A}_{L^{\infty}} \int_{B_{r}} \abs{\nabla w_{\pm}}^{2} \Phi_{\mL} \, \rmd x \le \frac{C}{r^{2}} \int_{B_{r}} \abs{x}^{2-n} \, \rmd x \le C,
\end{equation*}
then from \eqref{EQ:apriori-monotonicity2} and $\norm{A-\Id}_{L^{\infty}(\partial B_{r})} \le C r^{\alpha}$ we obtain
\begin{equation}
\begin{aligned}
& r \int_{\partial B_{r}} \nabla w_{\pm} \cdot A \nabla w_{\pm} \Phi_{\mL} \, \rmd \mH^{n-1} \\
& \quad = r \int_{\partial B_{r}} \abs{\nabla w_{\pm}}^{2} \Phi_{\mL} \, \rmd \mH^{n-1} + r \int_{\partial B_{r}} \nabla w_{\pm} \cdot (A-\Id) \nabla w_{\pm} \Phi_{\mL} \, \rmd \mH^{n-1} \\
& \quad \ge \gamma_{\pm} \left( \int_{\partial B_{r}} \abs{\hat{x}\cdot A\nabla (w_{\pm}^{2})} \Phi_{\mL} \, \rmd \mH^{n-1} + \int_{\partial B_{r}} w_{\pm}^{2} \abs{\hat{x} \cdot A \nabla \Phi_{\mL}} \, \rmd \mH^{n-1} \right) - C r^{2+\alpha} \\
& \quad \ge (2\gamma_{\pm} - Cr^{\alpha}) \ell_{\pm}(r). 
\end{aligned} \label{EQ:surface-estimate3}
\end{equation}

\medskip

\noindent \textbf{Step 3: Conclusion.} We now put the above estimates together to conclude our lemma. From \eqref{EQ:apriori-monotonicity2}, one sees that $\ell_{\pm}(r)$ is in $L^{1}$, and its derivative exists for almost all small $r$. By using \eqref{EQ:surface-estimate2} and \eqref{EQ:surface-estimate3}, there exists a positive constant $\epsilon>0$ such that    
\begin{equation*}
\begin{aligned}
& \frac{\rmd}{\rmd r} \left( \frac{1}{r^{4}} \ell_{+}(r) \ell_{-}(r) \right) \\
& \quad = - \frac{4}{r^{5}} \ell_{+}(r) \ell_{-}(r) + \frac{1}{r^{4}} \ell_{-}(r) \int_{\partial B_{r}} \nabla w_{+} \cdot A \nabla w_{+} \Phi_{\mL} \, \rmd \mH^{n-1} \\
& \qquad + \frac{1}{r^{4}} \ell_{+}(r) \int_{\partial B_{r}} \nabla w_{-} \cdot A \nabla w_{-} \Phi_{\mL} \, \rmd \mH^{n-1} \\
& \quad \ge \frac{1}{r} ( 2 \gamma_{+} + 2 \gamma_{-}) - 4 - Cr^{\alpha} ) \left( \frac{1}{r^{4}} \ell_{+}(r) \ell_{-}(r) \right) \\
& \quad \ge \frac{1}{r} ( 4\theta - Cr^{\alpha} ) \left( \frac{1}{r^{4}} \ell_{+}(r) \ell_{-}(r) \right) \ge \frac{\epsilon}{r} \left( \frac{1}{r^{4}} \ell_{+}(r) \ell_{-}(r) \right).
\end{aligned}
\end{equation*}
By integrating  the above inequality, we conclude our lemma. 
\end{proof}

\section*{Acknowledgments}

\noindent 
Kow was partly supported by the NCCU Office of research and development. Kow and Salo were partly supported by the Academy of Finland (Centre of Excellence in Inverse Modelling and Imaging, 312121) and by the European Research Council under Horizon 2020 (ERC CoG 770924). Shahgholian was supported by Swedish Research Council (grant no. 2021-03700). 

\section*{Declarations}

\noindent {\bf  Data availability statement:} All data needed are contained in the manuscript.

\medskip
\noindent {\bf  Funding and/or Conflicts of interests/Competing interests:} The authors declare that there are no financial, competing or conflict of interests.

\end{sloppypar}

\end{document}